\documentclass[10.5pt]{amsart}

\usepackage{geometry}              
\usepackage{graphicx}
\usepackage{amsmath,amsthm,amssymb}
\usepackage{epstopdf}
\usepackage{amssymb,amsfonts}

\usepackage[all,arc]{xy}
\usepackage{enumerate}
\usepackage{mathrsfs}
\usepackage{amsthm}
\usepackage{standalone}
\usepackage{tikz,pgf}
\usepackage{graphicx}
\usepackage{caption}

 \usepackage{amsaddr}

\usepackage{array}
\usepackage{enumitem}
\usepackage{scrextend}
\usepackage{tikz}
\usepackage{tikz-cd}
\usepackage{amsmath}
\usepackage{permute}
\usetikzlibrary{positioning}

\usepackage{clrscode3e}

\newtheorem{thm}{Theorem}[section]
\newtheorem{cor}[thm]{Corollary}
\newtheorem{prop}[thm]{Proposition}

\newtheorem{quest}[thm]{Question}

\newtheorem*{rem*}{Remark}
\theoremstyle{definition}
\newtheorem{defn}[thm]{Definition}

\newtheorem{example}[thm]{Example}

\newtheorem{rem}[thm]{Remark}

\numberwithin{equation}{section}

\title{Wirtinger systems of generators of knot groups}
\thanks{MSC codes: 57M25, 57M27, 57M05.}

\thanks{The first, third and fourth authors were partially supported by NSF grant DMS-1247679.}
\author{R. Blair, A. Kjuchukova, R. Velazquez, P. Villanueva}

\begin{document}
\setlength{\parindent}{0pt}

\begin{abstract} We define the {\it Wirtinger number} of a link, an invariant closely related to the meridional rank. The Wirtinger number is the minimum number of generators of the fundamental group of the link complement over all meridional presentations in which every relation is an iterated Wirtinger relation arising in a diagram. We prove that the Wirtinger number of a link equals its bridge number. This equality can be viewed as establishing a weak version of Cappell and Shaneson's Meridional Rank Conjecture, and suggests a new approach to this conjecture. Our result also leads to a combinatorial technique for obtaining strong upper bounds on bridge numbers. This technique has so far allowed us to add the bridge numbers of approximately 50,000 prime knots of up to 14 crossings to the knot table. As another application, we use the Wirtinger number to show there exists a universal constant $C$ with the property that the hyperbolic volume of  a prime alternating link $L$ is bounded below by $C$ times the bridge number of $L$.
  \end{abstract}
\maketitle

\section{Introduction}

This work was inspired by an old problem, 1.11 on the Kirby List~\cite{kirby1995problems}:

\begin{quest}
\label{MRC}
Is every knot whose group is generated by $2$ meridians actually an $2$-bridge knot? Same for $n$ meridians and $n$-bridge knots. 
\end{quest}

The above question has become known as the Meridional Rank Conjecture, and has been answered in the affirmative for many classes of links. The case of $n=2$ was settled in 1989 by Boileau and Zimmermann~\cite{BZ89}. The conjecture has also been shown to hold  for generalized Montesinos links~\cite{BZ85},~\cite{LM93}, torus links~\cite{RZ87}, iterated cables~\cite{CH14}, links of meridional rank 3 whose double branched covers are graph manifolds~\cite{BJW} and knots whose exteriors are graph manifolds~\cite{BDJW}. There are no known counter-examples, and the general case remains open.

The present work was inspired by the following simple observation about the conjecture. Denote the bridge number and meridional rank of a link $L$ by $\beta(L)$ and $\mu(L)$, respectively. Let us recall the classical argument which establishes $\beta(L)\geq\mu(L)$. Assume $\beta(L)=m$. Then, $L$ admits a diagram with exactly $m$ local maxima $x_1, ...,  x_m$ (with respect to some axis in the plane). The Wirtinger generators corresponding to the $m$ arcs containing the $x_i$ are then easily seen to generate the group of the link complement, by applying the Wirtinger relations in this diagram successively at crossings of decreasing height.

What is obvious yet intriguing about this argument is that it does not directly compare the bridge number to the number of meridional generators in a presentation of the link group in which arbitrary valid relations are allowed. Rather, only very particular, diagrammatic, relations are considered. This motivates studying the intermediate link invariant which arises by, intuitively speaking, considering only presentations with the property that the generators are  meridional elements and the relations are Wirtinger relations that can simultaneously be realized in a diagram.

To formalize this notion, we introduce the combinatorial tool of coloring a link diagram according to the following set of rules. Recall that if $L$ is a link in $\mathbb{R}^3$ and $p:\mathbb{R}^3\rightarrow \mathbb{R}^2$ is the standard projection map given by $p(x,y,z)=(x,y)$, then $p(L)$ is a \textit{link projection} if $p|_L$ is a regular projection. Hence a link projection is a finite four-valent graph in the plane, and we refer to the vertices of this graph as crossings. A \textit{link diagram} is a knot projection together with labels at each crossing that indicate which strand goes over and which goes under. By standard convention, these labels take the form of deleting parts of the under-arc at every crossing, and thus we think of a link diagram as a disjoint union of closed arcs, or \emph{strands}, in the plane, together with instructions for how to connect these strands to form a union of simple closed curves in $\mathbb{R}^3$.

Let $D$ be a diagram of a link $L$ with $n$ crossings.
Denote by $s(D)$ the set of strands $s_1$, $s_2$,..., $s_n$ and let $v(D)$ denote the set of crossings $c_1$, $c_2$,..., $c_n$. Two strands $s_i$ and $s_j$ of $D$ are {\it adjacent} if $s_i$ and $s_j$ are the under-strands of some crossing in $D$. There exists a unique knot diagram up to planar isotopy for which there exists a strand $s_i$ of $D$ for which $s_i$ is adjacent to itself, see Figure \ref{onecrossing.fig}. In all cases we consider, adjacent arcs are understood to be distinct.

\begin{figure}
	\includegraphics[width=2in]{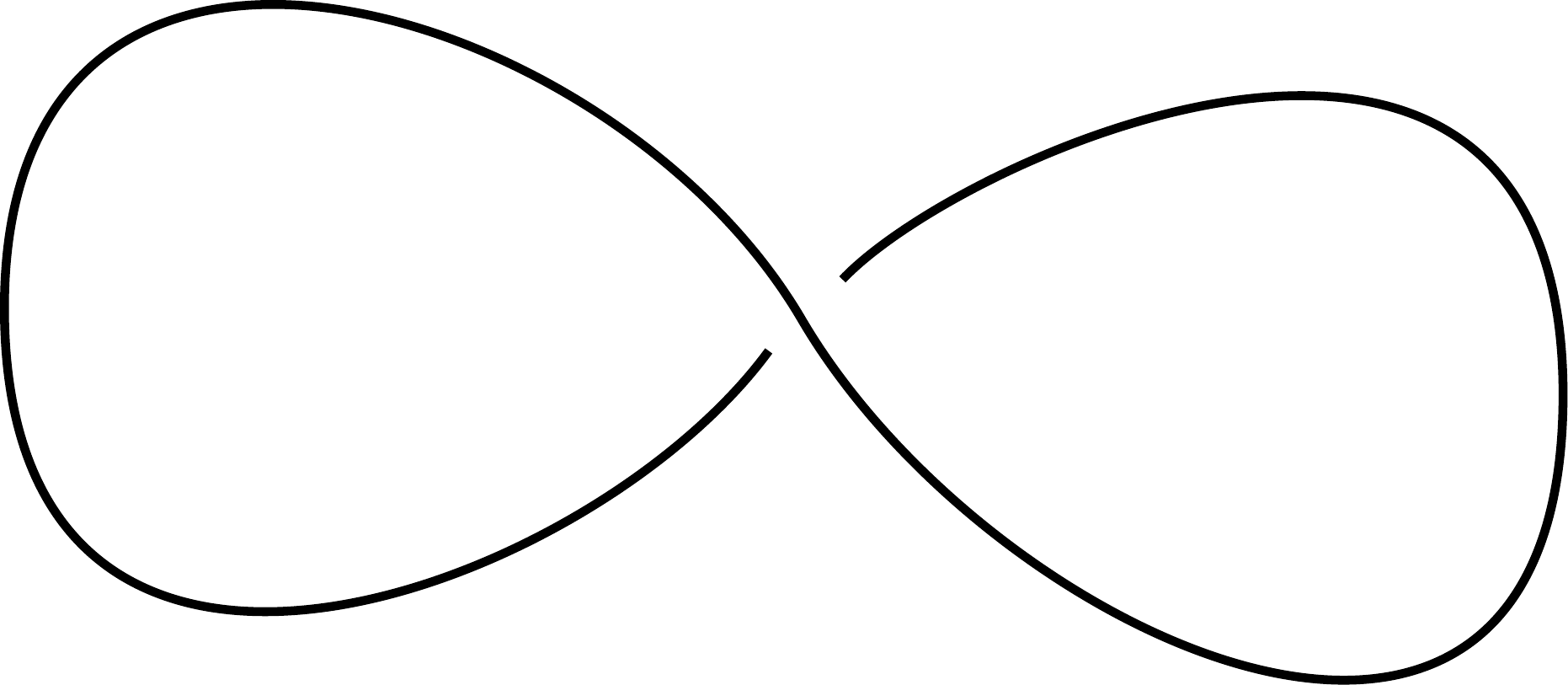}
	\caption{The only knot diagram in which a strand is adjacent to itself.}
	\label{onecrossing.fig}
	\end{figure}

We call $D$ \textit{k-partially colored } if we have specified a subset $A$ of the strands of $D$ and a function $f: A \to \{1, 2, \dots, k\}$.  We refer to this partial coloring by the tuple $(A, f)$. Given $k$-partial colorings $(A_1, f_1)$ and $(A_2, f_2)$ of $D$, we say $(A_2, f_2)$ is the result of a \textit{coloring move } on $(A_1, f_1)$ if
\begin{enumerate}
    \item $A_1 \subset A_2$ and $A_2 \setminus A_1 = \{s_j\}$ for some strand $s_j$ in $D$;
    \item $f_2|_{A_1} = f_1$;
    \item $s_j$ is adjacent to $s_i$ at some crossing $c \in v(D)$, and $s_i\in A_1$;
    \item the over-strand $s_k$ at $c$ is an element of $A_1$;
    \item $f_1(s_i) = f_2(s_j)$. 
\end{enumerate}
Denote the above coloring move from one k-partially colored diagram to another by $(A_1, f_1)\to(A_2, f_2)$. See Figure \ref{coloring.fig}.\footnote{Thanks to Patricia Cahn for creating Figure~2.} The move $(A_1, f_1)\to(A_2, f_2)$ captures the fact that if the Wirtinger generators corresponding to $s_i$ and $s_k$ belong to the subgroup of $\pi_1(S^3-L, x_0)$ generated by some set of meridians of $L$, then, by applying the Wirtinger relation at $c$, $s_j$ is seen to belong to this subgroup as well.

We say $D$ is \textit{k-meridionally colorable} if there exists a $k$-partial coloring $(A_0, f_0) = (\{s_{i_1}, s_{i_2}, \dots, s_{i_k}\}$, $f_0(s_{i_j}) = j)$ and a sequence of $c(D) - k$ coloring moves $(A_0, f_0)\to(A_1, f_1)\to\;\dots\;\to(A_{c(D) - k}, f_{c(D) - k})$, where $c(D)$ denotes the crossing number of $D$. In particular, $A_{c(D) - k}=s(D)$, that is, at the end of the coloring process every strand is assigned a color.  By design, the set $\{s_{i_1}, s_{i_2}, \dots, s_{i_k}\}$  corresponds to meridional elements that generate the link group via iterated application of the Wirtinger relations in $D$, so we refer to it as a {\it Wirtinger generating system}, and we call its elements  \textit{seed strands}. The minimum value of $k$ such that $D$ admits a Wirtinger generating system with $k$ elements (equivalently, $D$ is $k$-meridionally colorable) is the {\it Wirtinger number} of $D$, denoted $\omega(D)$. Of course, this number will depend on the choice of diagram, but it can be used to define an invariant of $L$.

\begin{figure}
	\includegraphics[width=4in]{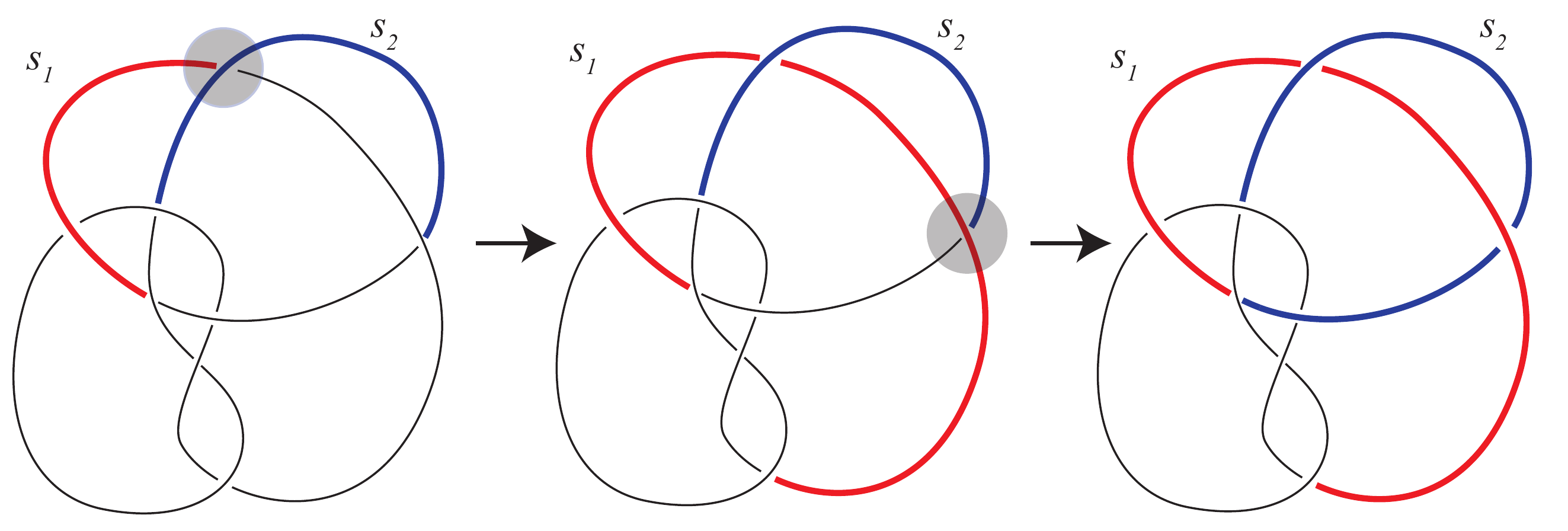}
	\caption{Two coloring moves on the knot $8_{17}$, corresponding to the shaded crossings. The coloring process terminates at this stage. More generally, this diagram can not be colored using only two seeds. $8_{17}$ is a three-bridge knot.}
	\label{coloring.fig}
	\end{figure}

\begin{defn}
\label{omega}  Let $L\subset S^3$ be a link. The {\it Wirtinger number} of $L$, denoted $\omega(L)$, is the minimal value of $\omega(D)$ over all diagrams $D$ of $L$.
\end{defn}

It is easy to see that $\omega(L)$ has the property $\mu(L)\leq\omega(L)\leq\beta(L)$. The first inequality follows from the fact that a Wirtinger generating system is by definition a meridional generating set. The second one is implied by the classical argument relating bridge number to meridional rank. Our main result is to show that  this inequality is in fact an equality.

\begin{thm}[Main Theorem]
\label{main}
Let $L\subset S^3$ be a link. The {\it Wirtinger number} and the bridge number of $L$ are equal.
\end{thm}

An immediate consequence of this result is that it provides a novel approach to computing bridge numbers of links. Although the question of finding the minimum of $\omega(D)$ over all diagrams $D$ of a given link $L$ is subtle, calculating $\omega(D)$ itself is algorithmic. This has allowed us to implement the calculation of $\omega(D)$ in Python. In the Appendix, we outline our algorithm for computing $\omega(D)$ from a Gauss code for $D$. The algorithm runs extremely fast in practice. We have used this computational approach to complete the tabulation of bridge number for all prime knots up to 12 crossings and the vast majority of knots with crossing number 13 and 14, thereby adding the bridge numbers of approximately 50, 000 knots to the knot table. Due to the fact that we can not assume that $\omega(L)$ is always realized in a minimal diagram of $L$ -- although this turns out to be the case for all prime knots of 12 crossings or less -- tabulating these bridge numbers requires proving that the upper bounds provided by the Wirtinger number are sharp. Our argument to this effect is presented in Section~\ref{app}, together with a discussion of our findings relating $\omega(L)$ to $\omega(D)$ where $D$ is a minimal diagram of $L$.

In the second place, the Wirtinger number allows us to relate $\beta(L)$ to other diagrammatic link invariants, such as the {\it twist number}. Recall that in the sphere of projection containing the link diagram, a {\it twist region} is either a maximal collection of bigons in the knot projection stacked end to end or a neighborhood of a crossing which is not contained in any bigon. The integer $t(D)$ denotes the number of twist regions of $D$. Lackenby~\cite{L04} showed that if a hyperbolic link has a prime alternating diagram $D$, then the hyperbolic volume of that link is bounded above and below by linear functions of $t(D)$. We can elevate $t(D)$ to a link invariant by declaring that $t(L)$ be equal to the minimum of $t(D)$ over all diagrams $D$ of $L$. We obtain:

\begin{cor}
\label{twist}
Given a link $L$, $\beta(L)\leq 2 t(L)$.
\end{cor}

Corollary~\ref{twist} has an immediate application to the the study of hyperbolic volumes of links. Closed 3-manifolds and link complements with a complete hyperbolic structure can be assigned a well-defined hyperbolic volume. The Heegaard genus of a closed 3-manifold $M$, denoted $g(M)$, is the minimum genus of any Heegaard surface for that manifold. Due to a theorem of Jorgensen and Thurston, there exists a constant $C$ such that if $M$ is a closed hyperbolic 3-manifold, then $Cg(M)\leq vol(M)$, where $vol(M)$ denotes the hyperbolic volume of $M$. Bridge number can be regarded as the analogue of Heegaard genus in the world of links. Recently, it was shown that there does not exist a $C$ such that for any hyperbolic link $C\beta(L)\leq vol(L)$ where $vol(L)$ is the hyperbolic volume of the complement of $L$~\cite{BTZ},~\cite{CFKNP}. It is a challenging open question to establish for what classes of links the analogue of Jorgensen and Thurston's theorem holds. As a consequence of Corollary~\ref{twist} and the main result from~\cite{L04}, we prove the following analogue of Jorgensen and Thurston's theorem for prime alternating hyperbolic links:

\begin{thm}
\label{twist2}
 There exists a universal constant $C$ with the property that every prime alternating hyperbolic link $L$ satisfies the inequality $C\beta(L)< vol(L)$.
\end{thm}

\section{Proof of the main theorem}

Let $D$ be a $k$-meridionally colorable diagram of some link $L$.  Our proof strategy will be to construct from $D$ a Morse embedding of $L$ into $\mathbb{R}^3$ with exactly $k$ local maxima. This will be carried out in two steps. First, we will study the process of extending a partial coloring of $D$ across the entire diagram. The purpose is to extract geometric information about the link $L$ from the sequence of coloring moves. Secondly, we will use the information obtained to construct the desired embedding.
It will prove useful to record the order in which strands are colored, as follows.

\begin{defn}
\label{height} Suppose $D$ is link diagram with crossing number $c(D)$. Assume $D$ can be $k$-meridionally colored by starting with a Wirtinger generating system $\{s_{i_1}, s_{i_2}, \dots, s_{i_k}\}$ and performing coloring moves $(A_0, f_0)\to(A_1, f_1)\to\;\dots\;\to(A_{c(D) - k}, f_{c(D) - k})$.  We associate to this succession of moves the \textit{coloring sequence} $\{\alpha_j\}_{j = 1}^{c(D)}$ given by $\alpha_j = s_{i_j}$ for $1 \leq j \leq k$ and $\alpha_j \in A_{j-k}\setminus A_{j - (k+1)}$ for $k + 1 \leq j \leq c(D)-k$. Furthermore, given a coloring sequence $\{\alpha_j\}$ we define its {\it height function} $h: s(D)\to \mathbb{Z}$ by $h(\alpha_j):=-j$.
\end{defn}

Introducing the negative sign here serves merely to indulge the authors' mild preference for focusing on local maxima, rather than local minima, in our construction. We also remark that any diagram $D$ of a  non-trivial link can give rise to a multitude of distinct coloring sequences. When a collection of seeds suffices to extend a partial coloring across all of $D$, the order in which moves are performed involves making arbitrary choices; the color a strand attains can also vary depending on the chosen order. However, once a succession of coloring moves is chosen, the associated coloring sequence is unique.

We review a couple of terms used in the proof of the next proposition. Let $A$ be some subset of $s(D)$.  We say the strands of $A$ are \textit{connected} if there exists a reordering of the strands $s_{i_1}, s_{i_2}, \dots, s_{i_a}$ in $A$ such that $s_{i_{j}}$ is adjacent to $s_{i_{j + 1}}$ for all $j, 1 \leq j \leq a - 1$.  Note the set of all strands in $D$ is connected if $K$ is a knot. In this case, if $A=s(D)$, then $s_{i_a}$ is adjacent to $s_{i_1}$.

Secondly, let $\{s_i\}^n_{i = 1}$ be a sequence of adjacent strands ordered by adjacency and let $g: \{s_1, s_2, \dots, s_n\} \to \mathbb{R}$ be a one-to-one map.  We say $g$ has a \textit{local maximum} at $s_j$ if the function $g': \{1, 2, \dots, n\} \to \mathbb{R}$ defined by $g'(i):=g(s_i)$ has a local maximum at $j$.

We now summarize the relevant properties of the functions $f_i$ and $h$. Because links require some additional considerations, we begin by studying the case of knots.

\begin{prop} \label{properties} Let $D$ be a diagram of a non-trivial knot. Assume $D$ can be $k$-meridionally colored via $(A_0, f_0)\to\dots\;\to(A_n, f_n)$, where $n={c(D) - k}$, and let $\{\alpha_m\}_{m = 1}^{c(D)}$ and $h$ be as above. The following hold:
\begin{enumerate}
\item For every $j \in \{1, 2, \dots, k\}, \delta \in \{0, 1, \dots, n\}, f^{-1}_\delta(j)$ is connected.
\item For any $j \in \{1, 2,\dots,k\}$, $h$ has a unique local maximum on $\{s_i|~s_i\in f^{-1}_n(j)\}$ when this set is ordered sequentially by adjacency. 
\item Let $s_p, s_q\in s(D)$ be adjacent understrands at a crossing in $D$, and denote the overstrand at this crossing by $s_r$. If $f_n(s_p) = f_n(s_q)$, then $h(s_r) > \min\{h(s_p), h(s_q)\}$.
\end{enumerate}
\end{prop}

\begin{proof}Since $D$ is the diagram of a nontrivial knot, whenever $s_p, s_q\in s(D)$ are adjacent understrands at a crossing in $D$, $s_p\neq s_q$. However, it is possible for the overstrand and an understrand at a crossing of $D$ to be the same strand (i.e. take $D$ to be the result of a type one Reidemeister move that increases crossing number).

{\it (1)} Colloquially, the assertion here is that at every stage $\delta$ of the coloring process, each color in the diagram corresponds to a connected arc of $K$. We verify this claim by induction on $u$, where  $u$ denotes the stage of the coloring process. For $u = 0$, $f^{-1}_0(j) = \{s_{i_j}\}$ is connected. Now assume $f^{-1}_u(j)$ is connected for all $u < t$, and let $u = t$ with $t>0$.  By definition of the coloring move, $\exists s_i \in s(D)$ such that $s_i$ is the unique strand to which a color is assigned at stage $t$. That is, $\{s_i\}= A_t\setminus A_{t - 1}$ and $s_i$ is adjacent to some $s_l \in A_{t-1}$.  Moreover, by definition of the coloring move, $f_t(s_i) = f_{t-1}(s_l)= c$.  Therefore,  $f^{-1}_t(c) = f^{-1}_{t-1}(c) \cup \{s_i\}$ is connected since $ f^{-1}_{t-1}(j)$ is connected by assumption, and $s_i$ is adjacent to a strand in $f^{-1}_{t-1}(c)$. Similarly, by definition of the coloring move, $\forall r\neq c$, $f^{-1}_t(r) = f^{-1}_{t-1}(r)$, which is connected by the inductive hypothesis.

{\it (2)} The statement is that $h$ attains a unique local maximum along each color; in fact, the local maximum  in every color is the seed strand. Intuitively, this follows from the fact that, by the definition of $h$, at every stage $u>0$ of the coloring process, the single strand $\{s_a\}= A_u\setminus A_{u-1}$ has the property that $h(s_a)=\min\{h(s_c)|~ s_c\in A_u\}$, so $s_a$ can not possibly introduce a new local maximum in its color. We formalize this argument by induction on $u$.  At stage $u = 0$, each color $j \in \{1, 2, \dots, k\}$ corresponds only to its seed strand. That is, $f^{-1}_0(j) =\{s_{i_j}\}$, and $h$ trivially attains a single local maximum on this set. Now assume that for $u<t$, $h$ has a unique local maximum on $f^{-1}_u(j)$. Set $u=t$. There exists a strand $\{s_i\}= A_t\setminus A_{t - 1}$ with the property that $s_i$ is adjacent to some $s_l \in A_{t-1}$ and $f_t(s_i)=f_t(s_l)=c$. But $h(s_i)=-t$ and $h(s_l)>-t$, since $s_l$ was colored before stage $t$ of the coloring process. Because $s_i$ and $s_l$ are adjacent, $s_i$ is not a local maximum in $f^{-1}_{t}(c)$ and the number of local maxima in each color remains unchanged.

{\it (3)} This claim can be rephrased by saying that if $s_p$ and $s_q$, two strands adjacent at a crossing, have been assigned the same color, then the over-strand $s_r$ at this crossing cannot have been the last one of the three to attain a color. The intuitive reason is that the definition of the coloring move dictates that $s_r$ must have been assigned a color in order for the coloring to be extended from $s_p$ to $s_q$ or vice-versa. To prove this assertion, let $h(s_c)= - \tau_c$ for $c\in\{p, q, r\}$ and assume for contradiction that $h(s_r) \leq \min\{h(s_p), h(s_q)\}$. Recall that $s_q \neq s_p$. Hence, without loss of generality, we can assume $\tau_p>\tau_q \geq \tau_r$. Denote $f_n(s_p)=f_n(s_q)=j$ and consider ${i_q}$ such that $\{s_q\}= A_{i_q}\setminus A_{i_{q}-1}$. By definition of the coloring move, at stage $\delta_{i_q}$, the color $j$ was extended to the strand $s_q$ from an adjacent strand $s_l$. That is, $\exists s_l\in A_{i_{q}-1}$ such that $s_l$ is adjacent to $s_q$ and $f_{i_{q}-1}(s_l)=j$. By assumption, $\tau_r\leq \tau_q$, so $s_r\notin A_{i_{q}-1}$. 
In particular, since $D$ is the diagram of a non-trivial knot, $s_l\neq s_p$. Moreover, $s_q$ is adjacent to both $s_l$ and $s_p$. Additionally, since $\tau_p>\tau_q$, we have $s_l, s_p\in f_{i_{q}-1}^{-1}(j)$ whereas $s_q\notin  f_{i_q-1}^{-1}(j)$. But we know from  {\it (1)} that $f_{i_{q}-1}^{-1}(j)$ is connected. Thus, since $D$ is the diagram of a knot, $f_{i_{q}-1}^{-1}(j)$ must contain all arcs in $D$ except $s_q$. If $s_r$ and $s_q$ are distinct strands, this contradicts the assumption that $s_r$ has not been colored by stage $i_{q}-1$. If $s_r=s_q$, it follows that, at stage $i_{q}$, the entire diagram is colored and $f(s_i)=j$ for all $i \in \{0, 1, \dots, n\}$. This implies that $\omega(D)=1$, so the meridional rank of $K$ is 1, contradicting our assumption that $D$ is a diagram of a non-trivial knot.

\end{proof}

\begin{rem*} The connectedness of $K$ plays an essential role in the proof of {\it (3)}, and this argument does not generalize without modification to the case of links. In fact, the Hopf link violates (3). \end{rem*}
In order to extend Proposition~\ref{properties} to links, we need to consider links which exhibit the above exception.
\begin{defn}
A link $L$ is $S^3$ is \emph{cut-split} if there exists an unknotted component $U$ of $L$ such that $U$ bounds an embedded disk $B^2$ in $S^3$ with $int(B^2)\cap L=\emptyset$ or $L$ meets $int(B^2)$ transversely in a single point. We call $U$ the \emph{splitting component} of $L$. A link diagram $D$ is \emph{cut-split} if there exists $s_p, s_q\in s(D)$ that are adjacent at some crossing of $D$ such that $s_p= s_q$ or if there exists an element of $s(D)$ that is a simple closed curve.\end{defn}

The standard diagram of the Hopf link is cut-split, with either of the link components as a splitting component. More generally, if $D$ is a cut-split diagram of a link $L$, then $L$ is cut-split. Indeed, a self-adjacent strand of $D$ corresponds to a component of $L$ that bounds a disk $B^2\subset S^3$ whose interior meets $L$ transversely in one or zero (see Figure~\ref{onecrossing.fig}) points. We leave it to the reader to verify the following easy facts about cut-split links and diagrams.

\begin{rem}\label{shortcut} Let $L\subset S^3$ be a link.

\begin{enumerate}
\item If $L$ is cut-split with splitting component $U$, then $\beta(L)=\beta(L\setminus U)+1$.
\item If $D$ is a cut-split link diagram of $L$, $U$ is the splitting component of $L$ that projects to the self adjacent strand $s_p$ or to a simple closed curve, and $D'$ is the the natural diagram of $L\setminus U$ corresponding to the removal of $s_p$ from $D$, then $\omega(D)=\omega(D')+1$.

\end{enumerate}

\end{rem}

\begin{prop} \label{propertieslinks} Let $D$ be a diagram of a link $L$ such that $D$ is not cut-split. Assume $D$ can be $k$-meridionally colored via $(A_0, f_0)\to\dots\;\to(A_n, f_n)$, where $n={c(D) - k}$, and let $\{\alpha_m\}_{m = 1}^{c(D)}$ and $h$ be as above. The following hold:
\begin{enumerate}
\item For every $j \in \{1, 2, \dots, k\}, \delta \in \{0, 1, \dots, n\}, f^{-1}_\delta(j)$ is connected.
\item For any $j \in \{1, 2,\dots,k\}$, $h$ has a unique local maximum on $\{s_i|~s_i\in f^{-1}_n(j)\}$ when this set is ordered sequentially by adjacency. (In the special case when $\{s_i|~s_i\in f^{-1}_n(j)\}$ is the set of all strands corresponding to the projection of a single component of $L$, this set is ordered cyclically by adjacency)

   \item Let $s_p, s_q\in s(D)$ be adjacent understrands at a crossing $c$ in $D$, and denote the overstrand at this crossing by $s_r$. If $f_n(s_p) = f_n(s_q)$, then one of the following holds:
    \begin{enumerate}

    \item $h(s_r) > \min\{h(s_p), h(s_q)\}$,
    \item the set $\{f^{-1}_n(f_n(s_p))\}$ corresponds to the projection of one component $U$ of $L$, and $c$ is the unique crossing incident to $p(U)$ with the property that $h(s_r) \leq \min\{h(s_p), h(s_q)\}$.

    \end{enumerate}

\end{enumerate}
\end{prop}

\begin{proof}
{\it (1)} and {\it (2)} follow without modification from the proof of Proposition \ref{properties} parts {\it (1)} and {\it (2)}.

{\it (3)} First, we reestablish the setup for the proof of Proposition \ref{properties} part {\it (3)}. Let $h(s_c)= - \tau_c$ for $c\in\{p, q, r\}$ and assume that $h(s_r) < \min\{h(s_p), h(s_q)\}$. Recall that $s_q \neq s_p$ since $D$ is not cut-split. Hence, without loss of generality, we can assume $\tau_p>\tau_q \geq \tau_r$. Denote $f_n(s_p)=f_n(s_q)=:j$ and let ${i_q}$ be such that $\{s_q\}= A_{i_q}\setminus A_{i_{q}-1}$. By definition of the coloring move, at stage $\delta_{i_q}$, the color $j$ was extended to the strand $s_q$ from an adjacent strand $s_l$. That is, $\exists s_l\in A_{i_{q}-1}$ such that $s_l$ is adjacent to $s_q$ and $f_{i_{q}-1}(s_l)=j$. By assumption, $\tau_r\leq \tau_q$, so $s_r\notin A_{i_{q}-1}$. 

If $s_l = s_p$, then $s_p$ is the only strand adjacent to $s_l$, and $\{s_p,s_q\}=\{f^{-1}_n(j)\}$ is the set of all strands corresponding to the projection of a single component $U$ of $L$. The projection $p(U)$ is then incident to exactly two crossings, $c$ and $c'$. Moreover, by the definition of coloring move, the overstrand at $c'$ is contained in $A_{i_{q}-1}$. This establishes that situation {\it (b)} described in the proposition holds.

Now assume that $s_l \neq s_p$ and note that $s_q$ is adjacent to both $s_l$ and $s_p$. Additionally, since $\tau_p>\tau_q$, as in the proof of Proposition~\ref{properties} (3), we have $s_l, s_p\in f_{i_{q}-1}^{-1}(j)$ whereas $s_q\notin  f_{i_q-1}^{-1}(j)$. But we know from  {\it (1)} that $f_{i_{q}-1}^{-1}(j)$ is connected. Thus, $f_{i_{q}-1}^{-1}(j)$ must contain all strands corresponding to the projection of a single component $U$ of $L$ except $s_q$. Hence, $i_q$ is the first stage of the coloring process at which every strand of $f_{n}^{-1}(j)$ is colored.

Assume that there exists a second crossing $c'$ such that  $h(s'_r) \leq \min\{h(s'_p), h(s'_q)\}$ where $s'_r$ is the overstrand at $c'$, $s'_p$ and $s'_q$ are adjacent understands at $c'$ and $s'_p, s'_q \in f_{n}^{-1}(j)$ (i.e. $s'_p$ and $s'_q$ are contained in the projection of $U$). By repeating the above argument for the strands incident to the crossing $c'$, at stage $\delta_{i'_q}$, the color $j$ was extended to the strand $s'_q$ from an adjacent strand $s'_l$ and $i'_q$ is the first stage of the coloring process at which every strand of $f_{n}^{-1}(j)$ is colored. Thus, $i'_q=i_q$ and $s'_q=s_q$. By definition, the strand $s_q$ is an understrand at the crossings $c$ and $c'$. Moreover, $s_q$ is an understrand at exactly two crossings, one of which has an uncolored overstrand at stage $i_q$ and one of which has a colored overstrand at stage $i_q$. Since both $c$ and $c'$ have uncolored overstands at stage $i_q$, then $c=c'$.

\end{proof}

\begin{proof}[Proof of the Main Theorem]

We prove the theorem by induction on $N$, the number of components of $L$.  

{\bf Step I.} Let $N=1$, that is, $L=:K$ is a knot. If $K$ is trivial, then $\omega(K)=\beta(K)=1$, so we can assume $K$ is non-trivial. It suffices to show that, if $K$ admits a diagram $D$ which is $k$-meridionally colorable, then $\beta(K)\leq k$.  We use Proposition~\ref{properties} to construct from $D$ a smooth embedding of $K$ in $\mathbb{R}^3$ with exactly $k$ local maxima.

We begin by embeding $D$ in the plane $z=-c(D)-2$ in $\mathbb{R}^3$. By assumption, $D$ can be $k$-meridionally colored via some succession of coloring moves $(A_0, f_0)\to\dots\;\to(A_n, f_n)$, where $n={c(D) - k}$, Let $\{\alpha_m\}_{m = 1}^{c(D)}$ be the associated coloring sequence and let and $h$ be its height function, as in Definition~\ref{height}. Note that the range of $h$ is the set $\{-1, -2, \dots, -c(D)\}$.

Next, embed a copy, denoted $\hat{s}_i$, of each strand $s_i$ of $D$ in the plane $z = h(s_i)$ in such a way that the orthogonal projection of $\mathbb{R}^3$ to the plane $z=-c(D)-2$ maps $\hat{s}_i$ to ${s}_i$. We call  $\hat{s}_i$  the {\it lift} of  ${s}_i$. In what follows, we show that the strands $\hat{s}_i$ can be connected in such a way that the resulting knot has $D$ as the diagram of its projection to the plane  $z=-c(D)-2$. That is, we construct arcs $s_{ij}$ in $\mathbb{R}^3$ connecting the lifts  $\hat{s}_i$ and  $\hat{s}_j$ of adjacent strands $s_i$, $s_j$ of $D$, in such a way that $(\cup_{r=1}^{c(D)}\hat{s}_r)\bigcup(\cup_{r=1}^{c(D)}s_{r r+1\mod c(D)})
\cong S^1$ and the  $s_{ij}$ correspond to the arcs of $K$ which are not visible in $D$, i.e. to the deleted underpasses at each crossing.

\textit{Set-up:} Let $c$ be an arbitrary crossing in $D$.  Label the overstrand at $c$ by $s_k$ and the understrands by $s_i$ and $s_j$, in some order. Pick a small $\varepsilon>0$ in such a way that the ball $B^2_\varepsilon(c)$ in the plane $z=-c(D)-2$ has non-trivial connected intersection with each strand $s_i, s_j, s_k$ and is disjoint from all other strands of $D$.  Consider the infinite cylinder $B^2_\varepsilon(c)\times\mathbb{R}$, where $\mathbb{R}$ denotes the $z$ direction. By construction, this cylinder intersects  $\hat{s}_i$, $\hat{s}_j$ and $\hat{s}_k$ and is disjoint from the lifts of the remaining strands. We will embed an arc $s_{ij}$ into $B^2_\varepsilon(c)\times\mathbb{R}$ in such a way that  $\hat{s}_i\cup s_{ij}\cup \hat{s}_j$ is a continuous arc and the orthogonal projection of  $\hat{s}_k\cup (\hat{s}_i\cup s_{ij}\cup \hat{s}_j)$ to the plane $z=-c(D)-2$ coincides with the corresponding section of $D$. (Informally, $s_{ij}$ will connect $\hat{s}_i$ to $\hat{s}_j$, and it will pass ``under" $\hat{s}_k$.)

\textit{Case 1:} Assume $f_n(s_i) = f_n(s_j)$, that is, at the end of the coloring process, the strands $s_i$ and $s_j$ are assigned the same color. Let $s_{ij}$ be a smooth, monotonically decreasing curve which connects the endpoints of $\hat{s}_i$ and $\hat{s}_j$ that are contained in the cylinder $B^2_\varepsilon(c)\times\mathbb{R}$ and which has the property that $s_{ij}$ itself is contained entirely within the cylinder.  Recall that, by Proposition~\ref{properties}, $h(s_k) > \operatorname{min}\{h(s_i), h(s_j)\}$. This implies that $s_{ij}$ can be chosen so that the orthogonal projection of $\hat{s}_k\cup (\hat{s}_i\cup s_{ij}\cup \hat{s}_j)$ to the plane $z=-c(D)-2$ is the subset of $D$, as desired. (Precisely, for any $\varepsilon_1\in(0, \varepsilon)$, one can guarantee that the intersection of $s_{ij}$ and half-space $z\geq h(s_k)$ is contained entirely outside the cylinder $B^2_{\varepsilon_1}(c)\times\mathbb{R}$.) See Figure~\ref{Case1.fig}.

\begin{figure}
	\includegraphics[width=1in]{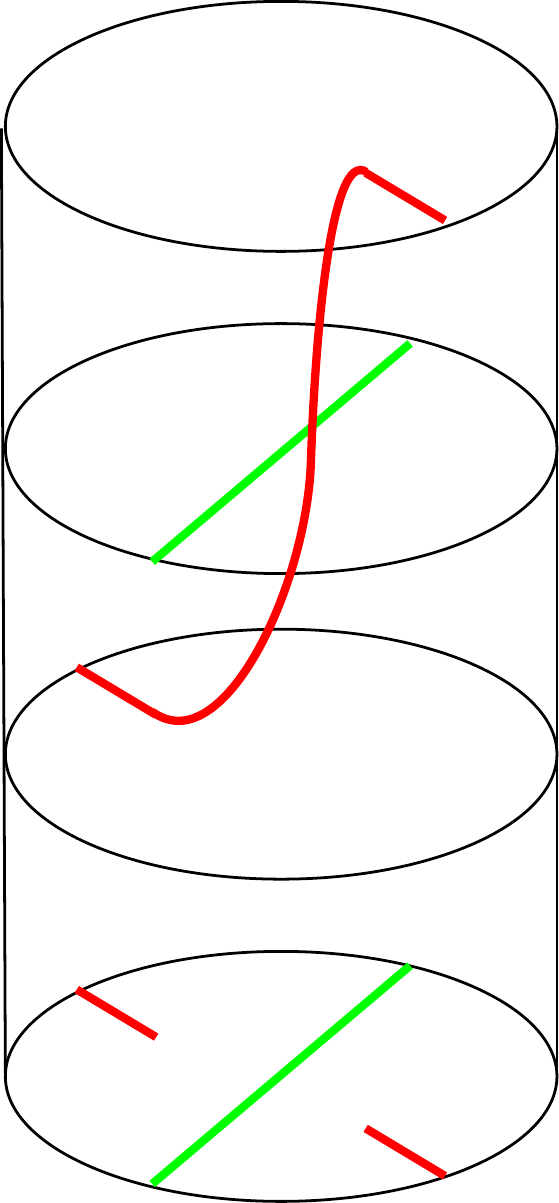}
	\caption{The construction of $s_{ij}$ in Case 1.}
	\label{Case1.fig}
	\end{figure}

\textit{Case 2:} Assume $f_{n}(s_i) \neq f_{n}(s_j)$, that is, at the end of the coloring process, the strands $s_i$ and $s_j$ are assigned distinct colors. Let $x_{ij}$ denote point in $(B^2_\varepsilon(c)\times\mathbb{R})\cap \{z=-c(D)-1\}$ with the property that the orthogonal projection of $x_{ij}$ to the plane $z=-c(D)-2$ coincides with the crossing  $c$ in the diagram. Construct $s_{ij}$ as the union of two smooth, monotonic arcs, contained entirely within $B^2_\varepsilon(c)\times\mathbb{R}$, connecting $x_{ij}$ to those endpoints of  $\hat{s}_i$ and $\hat{s}_j$ which are themselves contained in the cylinder. Because $h(s_k)\geq -c(D)>-c(D)-1$, these two monotonic arcs can be chosen so that the orthogonal projection of $\hat{s}_k\cup (\hat{s}_i\cup s_{ij}\cup \hat{s}_j)$ to the plane $z=-c(D)-2$ is once again a subset of $D$. (Precisely, for any $\varepsilon_1\in(0, \varepsilon)$, one can guarantee that the intersection of $s_{ij}$ and the cylinder $B^2_{\varepsilon_1}(c)\times\mathbb{R}$ is contained in $B^2_{\varepsilon_1}(c)\times[c(D)-1, c(D) - \frac{1}{2}]$.) See Figure \ref{Case2.fig}.

\begin{figure}
	\includegraphics[width=1in]{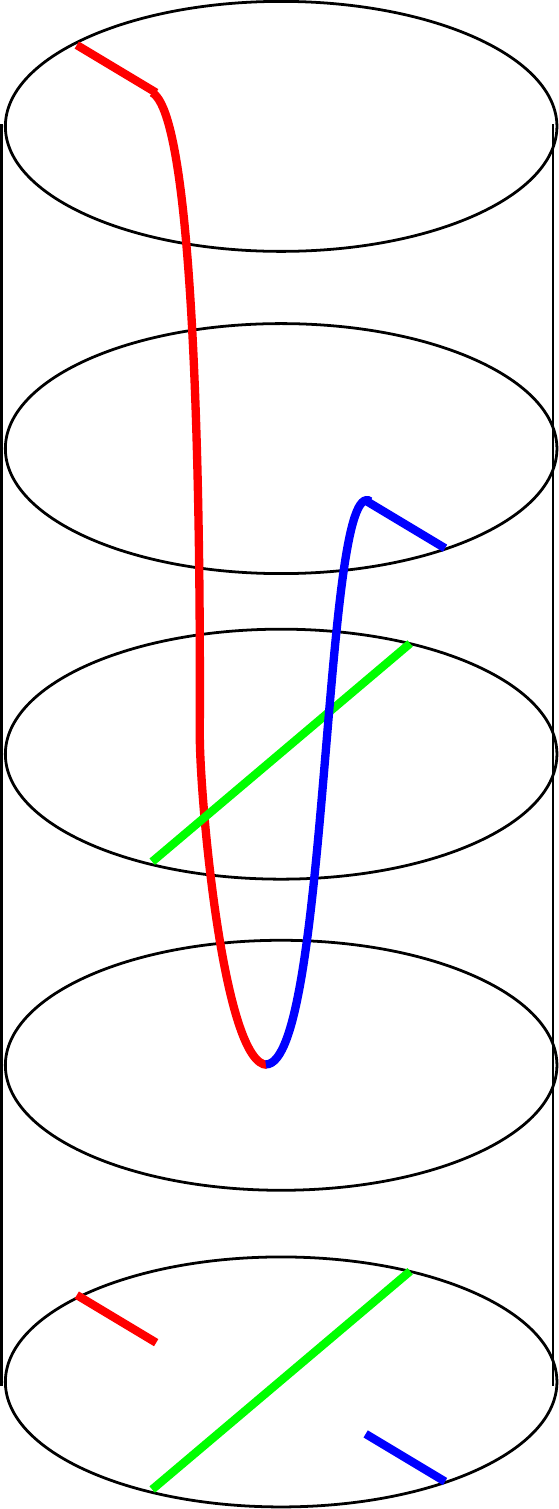}
	\caption{The construction of $s_{ij}$ in Case 2.}
	\label{Case2.fig}
	\end{figure}

In both Case 1 and Case 2, the above construction amounts to a careful way of joining a pair of adjacent strands in the diagram so that the overstrand at the crossing where they meet is preserved. Performing this construction at every crossing of $D$ therefore reconstructs an embedding of $K$. In order to produce from here an embedding with the desired number of local extrema, we perturb each lift  $\hat{s}_i$ to obtain a new lift $\hat{s}_i'$ in the following way. Let $c_{i i+1}$ denote the point in $s_{i i+1}$ that projects to a vertex of $p(K)$. If $h(s_i)$ is not the unique local maximum of $h$ on the set $f_n^{-1}(f_n(s_i))$, we let the subarc of $s_{i-1 i}\cup\hat{s}_i'\cup s_{i i+1}$ from $c_{i-1 i}$ to $c_{i i+1}$ be a smooth monotonic arc, strictly increasing or strictly decreasing as dictated by the values of $h(s_{i-1})$ and $h(s_{i+1})$.  On the other hand, if $h(s_i)$ is the unique local maximum of $h$ on the set $f_n^{-1}(f_n(s_i))$, we let the subarc of $s_{i-1 i}\cup\hat{s}_i'\cup s_{i i+1}$ from $c_{i-1 i}$ to $c_{i i+1}$ be a smooth arc increasing monotonically to the midpoint of $\hat{s}_i'$ and decreasing monotonically thereafter.

This construction produces a smooth embedding of $K$ in $\mathbb{R}^3$ with exactly $k$ local maxima, corresponding to the seed strands in each color. The local minima correspond to the points $x_{ij}$, and thus project to those crossings in $D$ at which the diagram changes color.

{\bf Step II.}  Now let $L$ be a link of $N>1$ components, and assume that $\omega(L')=\beta(L')$ for all links $L'$ of fewer than $N$ components. First consider the case that $L$ has a diagram with $\omega(D)=\omega(L)$ and such that $D$ is not cut-split. In this situation, Proposition~\ref{propertieslinks} applies. A minor adaptation of the proof given in Part~I will establish that $\omega(L)=\beta(L)$. We adopt the identical setup and begin by re-examining the cases.

\textit{Case 1:} Assume $f_n(s_i) = f_n(s_j)$ and $h(s_k) > \operatorname{min}\{h(s_i), h(s_j)\}$. Construct $s_{ij}$ exactly as in Case~1 of Step~I.  

\textit{Case 2:} Assume $f_{n}(s_i) \neq f_{n}(s_j)$. Construct $s_{ij}$ exactly as in Case~2 of Step~I. 

\textit{Case 3:} Assume $f_n(s_i) = f_n(s_j)$ and $h(s_k) \leq \operatorname{min}\{h(s_i), h(s_j)\}$. By Proposition~\ref{propertieslinks}, the set $\{f^{-1}_n(f_n(s_i))\}$ corresponds to the projection of a single component $U$ of $L$ and $c$ is the unique crossing incident to $p(U)$ with the property that $h(s_k) \leq \operatorname{min}\{h(s_i), h(s_j)\}$. Then, exactly as in Case 2 of Proposition \ref{main}, we construct $s_{ij}$ as the union of two smooth, monotonic arcs, connecting $x_{ij}$ to endpoints of  $\hat{s}_i$ and $\hat{s}_j$. Moreover, these two monotonic arcs can be chosen so that the orthogonal projection of $\hat{s}_k\cup (\hat{s}_i\cup s_{ij}\cup \hat{s}_j)$ to the plane $z=-c(D)-2$ is once again a subset of $D$. Note that $U$ is monochromatic, and the arc $s_{ij}$ contains the unique local minimum in this color of the constructed embedding.

Performing the above construction at every crossing of $D$ reconstructs an embedding of $L$. As in Step I, we can perturb this embedding slightly to produce a smooth embedding of $L$ in $\mathbb{R}^3$ with exactly $\omega(D)$ local maxima, corresponding to the seed strands in each color. This completes the proof of the proposition in the case when $L$ is a link with a diagram $D$ that is not cut-split and with $\omega(D)=\omega(L)$.

Now allow $L$ to be an arbitrary link of $N$ components and let $D$ be a diagram of $L$ such that $\omega(L)=\omega(D)$. If $D$ is not cut-split, then $\omega(L)=\beta(L)$ by our previous argument. Hence, we can assume both $D$ and $L$ are cut-split with splitting component $U$. Recall that, by Remark~\ref{shortcut}, $\beta(L)=\beta(L\setminus U)+1$. Moreover, $U$ is the splitting component of $L$ that projects to a self-adjacent strand or a simple closed curve in $s(D)$. By Remark~\ref{shortcut},  if $D'$ is the the natural diagram of $L\setminus U$ corresponding to the removal of the self-adjacent strand or simple closed curve from $D$, then $\omega(D)=\omega(D')+1$. Hence, $\omega(L\setminus U)\leq \omega(D)-1=\omega(L)-1$. By the induction hypothesis, we have $\omega(L\setminus U)=\beta(L\setminus U)=\beta(L)-1$. Thus, $\beta(L)-1= \omega(L\setminus U)\leq \omega(L)-1$. Since $\omega(L)\leq \beta(L)$ for all $L$, it follows that $\beta(L)=\omega(L)$, completing the proof.

\end{proof}

\section{Applications and further questions}
\label{app}

We begin with the proofs of Corollary~\ref{twist} and Theorem~\ref{twist2}, which relate the bridge number to the twist number of links and the hyperbolic volume of prime alternating links. Subsequently, we discuss applications of Theorem~\ref{main} to the tabulation of bridge number. 

\begin{proof}[Proof of Corollary~\ref{twist}]

Given a link $L$, it follows from the definition of $t(L)$ that $L$ admits a diagram $D$ with exactly $t(L)$ twist regions. These twist-regions are connected via  $2t(L)$ strands to form $D$. In other words, there are at most $2t(L)$ strands in $D$ which are not properly contained in a twist region. Declare each of these $2t$ strands to be a seed strand. This defines a $2t(L)$-partial coloring of $D$ with the property that all four strands of $D$ incident to the boundary of any given twist region of $D$ have received a color. Recall that a twist region constitutes either a single crossing or a collection of bigons. Therefore, the coloring move by definition allows us to extend the coloring of strands incident to the boundary of a  twist region across the entire region.  It follows that $\beta(L)=\omega(L)\leq\omega(D)\leq 2t(L)$, as claimed. 
\end{proof}

\begin{proof}[Proof of Theorem~\ref{twist2}] Let $L$ be a prime alternating link and let $D$ be a reduced alternating diagram for $L$. By~\cite{L04}, $\frac{1}{2}v_3(t(D)-2)\leq vol(L)$, where $v_3$ is the volume of a regular hyperbolic ideal 3-simplex. By Corollary~\ref{twist}, $\beta(L)\leq 2 t(L)$. Hence, $\frac{1}{2}v_3(\frac{1}{2}\beta(L)-2)\leq vol(L)$. In order to eliminate the constant term in the previous inequality, we note that Cao and Meyerhoff have shown that the minimum volume of any hyperbolic knot is $2v_3$ \cite{CM01}. Hence, we can set $C=\frac{1}{6}v_3$ to insure that $C\beta(L)\leq max(\frac{1}{2}v_3(\frac{1}{2}\beta(L)-2), 2v_3)\leq vol(L)$ for all values of $\beta(L)$.
\end{proof}

The Wirtinger number can also be used to compute bridge numbers of links.  Since for any diagram $D$ of a link $L$ we have $\omega(D)\geq \omega(L)=\beta(L)$, the Wirtinger number provides an approach to calculating upper bounds on the bridge number of $L$.
Moreover, as previously noted, the $\omega(D)$ for a given diagram $D$ is readily computed; for knot diagrams, this can be done via the computer algorithm outlined in the Appendix. The upper bounds obtained in this manner have turned out to be astonishingly strong. At the start of this project, according to KnotInfo~\cite{knotinfo}, bridge numbers were tabulated for prime knots up to and including 11 crossings. By comparing these bridge numbers to the Wirtinger numbers of minimal diagrams, we verified that for all prime knots with up to 11 crossings the upper bounds on $\beta$ obtained by computing $\omega(D)$ for representative minimal diagrams are sharp. 

We computed Wirtinger numbers for minimal diagrams of all 12-, 13- and 14- crossing knots as well. The number of knots among them whose minimal diagrams have Wirtinger number 2 coincided exactly with the number of two-bridge knots of 12, 13 and 14 crossings~\cite{de20072}. Therefore, our calculations identify all two-bridge knots in this range. It also follows that all diagrams $D$ with $\omega(D)=3$ represent three-bridge knots. To complete the tabulation of bridge number for prime 12-crossing knots, we calculated that all such knots have Wirtinger number at most 4, and we checked that the knots whose minimal diagrams have Wirtinger number 4 are not three-bridge. This was done by hand using methods of Jang~\cite{jang2011classification}. (We believe that the same method would allow us to complete the tabulation of bridge number for knots of 13 crossings as well.) Altogether, our computations so far have newly determined the bridge number of approximately 50,~000  prime knots of less than 15 crossings. In addition, as a corollary of these computations, we have verified that for all prime knots of less than 13 crossings, the Wirtinger number of some minimal diagram realizes the bridge number.  We propose the following:

\begin{quest}
\label{minimal}
{\it(Property M)} For which links $L$ is $\omega(L)$ realized in a minimum-crossing diagram of $L$?
\end{quest}

Our calculations show that all prime knots of up to and including 12 crossings have Property M. We conjecture that all prime 13-crossing knots do as well, and that the Wirtinger numbers found equal the bridge numbers.  Furthermore, by taking connected sums of two-bridge knots, one can construct families of knots which have Property M and whose crossing number and bridge number are unbounded. The question of completely characterizing knots with Property M remains open.

Let us now turn an eye back to the Meridional Rank Conjecture. The main theorem of this paper reduces the Conjecture to the following:

\begin{quest}
\label{reduce}
Does every link admit a minimal meridional presentation in which all relations arise as iterated Wirtinger relations in a diagram?
\end{quest}

A positive answer to this question for a class of links would mean that $\mu(L)= \omega(L)$, which, together with our result $\beta(L)=\omega(L)$, would imply the conjecture for these links. In particular, our point of view casts the Meridional Rank Conjecture as a question about the type of relations in a meridional presentation.

\section{Appendix: Computing $\omega(D)$}

We sketch the algorithm by which we obtained the computational results discussed previously. From now on we work only with knots. Furthermore, we make the following simplifying assumption. Note that coloring a knot diagram $D$ in several colors allowed us to study the combinatorics of the coloring process, which in turn enabled us to count the number of local maxima in the knot embedding we reconstructed from $D$. This analysis is a bit more subtle than what we need if we are merely asking whether a set $A$ of meridional elements generates the knot group via iterated application of the Wirtinger relations in $D$. Therefore, for the purpose of  calculating $\omega(D)$, we do not keep track of the different colors. Instead, we simply ask if a given partial coloring of $D$ can be extended to all of $D$. (Formally, we compose the function $f: A \to \{1, 2, \dots, k\}$ with the constant function $c:  \{1, 2, \dots, k\}\to \{1\}$, then we define the coloring move as before.) The algorithm can be broken down into three steps.
\begin{enumerate}
\item From the Gauss code of a non-trivial knot diagram $D$, extract information about which strands are over- and under-strands at every crossing of $D$.
\item Given a subset $A$ of set of strands $s(D)$, determine if choosing the strands in $A$ as seeds would allow the entire diagram to be colored by iterating the coloring move. 
\item Running across all subsets of size $k\geq 2$ of $s(D)$, determine if $D$ admits a Wirtinger generating system of size $k$. The algorithm terminates as soon as the first valid coloring occurs.
\end{enumerate}

Now we describe in some detail how these steps are performed. 

{\it (1) Creating a knot dictionary.} Let $K$ be a knot with diagram $D$.  By convention, we label the strands of $D$ by letters.  Represent each crossing of $D$ by the (unordered) tuple $(a, b)$, where $a$ and $b$ are the understrands at that crossing.  The \textit{knot dictionary} $D_K$ is a map which assigns to each element $c$ of $s(D)$ a subset of the crossings of $D$. The map is given by $D_K(c) = \{ (a, b)\mid c \text{ is the overstrand of }(a, b)\}$.In terms of data structures, the knot dictionary is a map whose keys are the strands of the knot diagram and whose values are subsets of $v(D)$.
\begin{example} The trefoil has knot dictionary $D_{3_1} =\{a \to \{(b, c)\}, b \to \{(a, c)\}, c \to \{(a, b)\}\}$.
\end{example}
We can derive the knot dictionary $D_K$ of a knot $K$ by examining its Gauss code $G_K$, using a function we call $\proc{knot-dictionary}(G_K)$. To illustrate how this function works, we return to the diagram of the trefoil, which has Gauss code $G_{3_1} = [-1, 3, -2, 1, -3, 2]$.  Since the negative numbers in the Gauss code correspond to a strand going under a crossing, we see that each strand is described by a subsequence of $G_{3_1}$ beginning and ending with a negative number (``wrapping around" the sequence if needed).  Since there are three strands in this diagram of the trefoil, there are three corresponding subsequences of $G_{3_1}$, which we have labeled to be consistent with the knot dictionary representation in the example above.  These three subsequences are $a = [-1, 3, -2], b = [-2, 1, -3]$, and $c = [-3, 2, -1]$.

Once we have determined which subsequences correspond to which strands, we next determine the crossings at which they are overstrands. We do this by examining the positive integers in each subsequence.  For example, since $a = [-1, 3, -2]$, the strand labeled $a$ is the overstrand at the crossing labeled $3$.  Then, since $b$ and $c$ contain $-3$, this indicates that they are under this same crossing and are thus the two strands under strand $a$ at crossing $3$.  We then assign the tuple $(b, c)$ to $a$.  Since $a$ contains no more positive integers, we have found all the crossings $a$ is over and have completed the knot dictionary entry for $a$, which is $D_{3_1}(a) \to \{(b, c)\}$, as in our above example.  Repeating this process for the remaining subsequences results in the same knot dictionary as in our above example: $D_{3_1} =\{a \to \{(b, c)\}, b \to \{(a, c)\}, c \to \{(a, b)\}\}$. 

{\it (2) Extending a partial coloring.}
Once we have a knot dictionary $D_K$, we can determine whether a given set of seed strands $A$ leads to a coloring of every strand in the diagram. We do this using a function called $\proc{color}$. Consider a crossing $(p, q)$ of the diagram and assume that $p$ is not colored. The partial coloring can be extended at this crossing if and only if both $q$ and the overstrand are colored. Running through the list of crossings of $D$ in any order allows us to determine if a coloring move can be performed.

The function $\proc{color}$ works as follows. Make a copy, $C$, of the seed strands $A$, and iterate through the keys of $D_K$ that are in $C$.  For each of these keys in $C$, say $a$, examine each crossing in $D_K(a)$. For each $(b, c)\in D_K(a)$,  if $C$ contains either of $b$ or $c$, add the other one to $C$.  Repeat this step until either all the strands of $D$ are added to $C$ (in which case, $A$ has been shown to be a Wirtinger generating system) or one entire iteration through all $a\in C$ and all $(b, c)\in D_K(a)$ is completed without adding new strands to $C$ (in which case, $A$ has been shown to not be a Wirtinger generating system).  

{\it (3) Finding a minimal coloring.}
We define a function $\proc{calculate-w}$ which determines the Wirtinger number for a knot diagram $D$. Given Gauss code $G_K$, we first call $\proc{knot-dictionary}(G_K)$ to create the corresponding knot dictionary $D_K$.  Then, for $n$ ranging from 1 to $D_K.size$ (the number of keys in $D_K$, i.e., the number of strands in the knot diagram), we repeat the following: we call $\proc{combinations}(D_K.keys, n)$, which returns $X$, the set of all combinations of $n$ strands; then, for each set of seed strands $A\in X$, we call $\proc{color}(D_K, A)$.  If $\proc{color}(D_K, A)$ results in coloring the entire diagram, we return $A.size$, the number of strands in $A$.  Otherwise, we pick a new set of seed strands $A'$ from $X$ and repeat this process.  If none of the combinations in $X$ lead to a complete coloring, we increment $n$ and repeat the process until such a combination is found.  Note that every non-trivial knot diagram is colorable by $c(D) - 1$ strands, so this algorithm is guaranteed to return $D_K.size - 1$ in the worst case. 

Since the function  $\proc{color}$ is applied to every subset of $s(D)$ of a given size $k\leq\omega(D)$, the algorithm runs in factorial time. However,  $\omega(D)<<c(D)$ in general, and the algorithm terminates when the first valid coloring occurs. As a result,  the running time is short in practice. Computing the Wirtinger numbers of all  diagrams in the Knot Table of up to 14 crossings took approximately 10 minutes on a weak fashionable laptop. That said, it is evident that the algorithm performs many redundant checks, and its efficiency can definitely be improved, should the running time increase unreasonably with $c(D)$.  We also remark that the above procedure for calculating $\omega(D)$ can be extended to link diagrams, by implementing a few modifications to handle Gauss code for multiple-component links.

\section*{Acknowledgement}
The authors would like to thank Michel Boileau for suggesting the name {\it Wirtinger number}.

\nocite{RZ87}
 \bibliographystyle{plain}
\bibliography{wirtinger}

\end{document}